\documentclass{amsart}
\pdfoutput=1
\title{Unital $k$-restricted $\infty$-operads}
\author{Amartya Shekhar Dubey}
\address{School of Mathematical Sciences, National Institute of Science Education and Research, Bhubaneswar, 752050}
\email{amartyashekhar.dubey@niser.ac.in}
\author{Yu Leon Liu}
\address{Department of Mathematics, Harvard University, 1 Oxford St, Cambridge, MA 02139}
\email{yuleonliu@math.harvard.edu}

\numberwithin{equation}{section}
\usepackage[utf8]{inputenc}
\usepackage[T1]{fontenc}
\usepackage{lmodern}
\usepackage[margin=1.2in]{geometry}
\usepackage{setspace}
\usepackage[usenames, dvipsnames]{xcolor}
\usepackage{graphicx}
\usepackage{mathtools}
\usepackage{amssymb}
\usepackage{amsthm}
\usepackage{xspace}
\usepackage{tikz-cd}
\usepackage{tkz-euclide}
\usepackage{tkz-graph}
\usetikzlibrary{
  knots,
  hobby,
  decorations.pathreplacing,
  shapes.geometric,
  calc
}
\usepackage{bbm}
\usepackage{slashed}
\usepackage[backref=page, bookmarks=false]{hyperref}
\usepackage[capitalize, noabbrev]{cleveref}
\newcommand{\op}{\mathsf{op}}

\newsavebox{\pullback}
\sbox\pullback{%
\begin{tikzpicture}%
\draw (0,0) -- (1ex,0ex);%
\draw (1ex,0ex) -- (1ex,1ex);%
\end{tikzpicture}}

\DeclareMathOperator{\Hom}{Hom}

\newtheorem{lemma}[equation]{Lemma}
\newtheorem{corollary}[equation]{Corollary}
\newtheorem{proposition}[equation]{Proposition}

\newtheorem{theorem}[equation]{Theorem}

\newtheorem*{lemma*}{Lemma}

\theoremstyle{definition}
\newtheorem{example}[equation]{Example}
\newtheorem{definition}[equation]{Definition}

\newtheorem{observation}[equation]{Observation}
\newtheorem{notation}[equation]{Notation}

\theoremstyle{remark}
\newtheorem{remark}[equation]{Remark}

\crefname{thm}{Theorem}{Theorems}
\crefname{lem}{Lemma}{Lemmas}
\crefname{cor}{Corollary}{Corollaries}
\crefname{prop}{Proposition}{Propositions}
\crefname{ex}{Exercise}{Exercises}
\crefname{exm}{Example}{Examples}
\crefname{defn}{Definition}{Definitions}
\crefname{claim}{Claim}{Claims}
\crefname{rem}{Remark}{Remarks}
\crefname{fct}{Fact}{Facts}
\crefname{note}{Note}{Notes}

\DeclarePairedDelimiter\paren{(}{)}

\makeatletter
	\let\oldparen\paren
	\def\paren{\@ifstar{\oldparen}{\oldparen*}}
\makeatother

\usepackage{microtype}
\usepackage{hypcap}

\hypersetup{
 colorlinks,
 linkcolor={red!50!black},
 citecolor={green!50!black},
 urlcolor={blue!80!black}
}

\newcommand{\cO}{\mathcal{O}}

\newcommand{\cD}{\mcal{D}}

\usepackage{autonum} 

\newcommand{\Fun}{\mathrm{Fun}} 

\newcommand{\Opone}{\mathrm{Op_1}}
\newcommand{\Op}{\mathrm{Op}}
\newcommand{\un}{\mathrm{un}}

\newcommand{\Fact}{\mathrm{Fact}}

\newcommand{\NU}{\mathrm{nu}}
\newcommand{\Triv}{\mathrm{Triv}}

\newcommand{\Mul}{\mathrm{Mul}}
\newcommand{\Opnu}{\Op^{\NU}}

\newcommand{\Opun}{\Op^{\un}}
\newcommand{\Omegac}{\Omega^c}
\newcommand{\cPseg}{\cP^{\mathrm{seg}}}

\newcommand{\AAnu}{\AA^{\mathrm{nu}}}

\newcommand{\Opunleq}[1]{\Opun_{\leq {#1}}}

\newcommand{\cat}{\mathrm{Cat}_{\infty}}

\newcommand{\Opunleqk}{\Opunleq{k}}

\newcommand{\Omegacleq}[1]{\Omegac_{\leq {#1}}}
\newcommand{\CC}[1]{\overline{C_{#1}}}

\newcommand{\leqk}{{\leq k}}

\newcommand{\Omegacleqk}{\Omegacleq{k}}

\newcommand{\Spaces}{\mathcal{S}}

\newcommand{\Lk}{\mathrm{L}_k}
\newcommand{\Rk}{\mathrm{R}_k}

\newcommand{\Seg}{\mathrm{Seg}}
\newcommand{\cpl}{\mathrm{cpl}}

\newcommand{\Free}{\mathrm{Free}}
\newcommand{\Segcpl}{\Seg^{\cpl}}
\tikzset{%
    symbol/.style={%
        draw=none,
        every to/.append style={%
            edge node={node [sloped, allow upside down, auto=false]{$#1$}}}
    }
}

\newcommand{\maxsurj}{\mathrm{max-surj}}
\newcommand{\rmax}{\mathrm{rmax}}

\def\AA{\mathbb A}
\def\EE{\mathbb E}

\def\cC{\mathcal C}\def\cD{\mathcal D}
\def\cF{\mathcal F}

\def\cO{\mathcal O}\def\cP{\mathcal P}

\newcommand{\LKan}{\mathrm{LKan}}
\newcommand{\RKan}{\mathrm{RKan}}
\newcommand{\inc}{\mathrm{inc}}
\newcommand{\glue}{\mathrm{glue}}
\newcommand{\pt}{\mathrm{pt}}
\newcommand{\rsub}{\mathrm{rsub}}
\newcommand{\rsubT}{{^\rsub/T}}
\newcommand{\el}{\mathrm{el}}

\makeatletter
\newcommand{\colim@}[2]{%
  \vtop{\m@th\ialign{##\cr
    \hfil$#1\operator@font colim$\hfil\cr
    \noalign{\nointerlineskip\kern-\ex@}\cr}}%
}
\newcommand{\colim}{%
  \mathop{\mathpalette\colim@{\scriptscriptstyle}}\nmlimits@
}

\newcommand{\Omegacrsub}{\Omega^{\mathrm{c}, \rsub}}
\newcommand{\FM}{\mathrm{FM}}

\begin{document}

\begin{abstract}
We study unital $\infty$-operads by their arity restrictions.
Given $k \geq 1$, 
we develop a model for unital $k$-restricted $\infty$-operads, which are variants of $\infty$-operads which have only $(\leq k)$-arity morphisms, as complete Segal presheaves on closed $k$-dendroidal trees, which are closed trees built from corollas with valences $\leq k$. Furthermore, we prove that the restriction functors from unital $\infty$-operads to unital $k$-restricted $\infty$-operads admit fully faithful left and right adjoints by showing that the left and right Kan extensions preserve complete Segal objects. Varying $k$, the left and right adjoints give a filtration and a co-filtration for any unital $\infty$-operads by $k$-restricted $\infty$-operads. 
\end{abstract}
\maketitle

\tableofcontents

\section{Introduction}
In the pioneering work \cite{stasheff1963homotopy}, Stasheff constructs 
a sequence of convex polytopes $K_n$, called the \emph{Stasheff associahedra}, that encode the associativity of a binary multiplication.
In \cite[\S 4.1]{HA}, Lurie provides an interpretation using the framework of $\infty$-operads. Let $\AAnu_\infty$ be the non-unital associative $\infty$-operad that parametrizes a fully coherent associative multiplication. Lurie constructs a converging filtration of $\infty$-operads
\begin{equation}\label{eq:non-unital-AAn-adjoint}
  \Triv = \AAnu_1 \to \AAnu_2 \to \cdots \to \AAnu_k \to \cdots \to \AAnu_\infty,
\end{equation}
where the $\AAnu_k$ 
parametrizes multiplications that are associative up to $k$ inputs. 
Furthermore, by \cite[Theorem 4.1.6.13]{HA}, extending an $\AAnu_{k-1}$-algebra structure to an $\AAnu_{k}$-algebra structure is equivalent to lift certain maps from $\partial K_k$ to $K_k$.

There is a sense that $\AAnu_k$ is the ``$k$-th arity restriction'' of $\AAnu_\infty$: 
\begin{enumerate}
    \item The map $\AAnu_k \to \AAnu_\infty$ induces an equivalence on $n$-ary morphism spaces for $n \leq k$.
    \item For $n > k$, the $n$-ary morphism space $\AAnu_k(n)$ is generated by $(\leq k)$-morphisms in a suitable sense.
\end{enumerate}
This motivates the natural notion of a $k$-restricted $\infty$-operad, where we only consider $n$-ary morphism spaces for $n \leq k$.
These $k$-restricted $\infty$-operads can be viewed as the $k$-arity-skeletons of $\infty$-operads, and the associated obstruction theory has been studied in \cite{barkan2023}.

In \cite[Appendix C.1]{MR4320769}, Heuts considers arity restrctions of non-unital $\infty$-operads.\footnote{What we call $k$-restricted is called $k$-truncated in \cite{MR4320769}.}\footnote{An $\infty$-operad $\cO$ is \emph{non-unital} if for every color $X$ the space $\Mul_{\cO}(\varnothing, X)$ is empty.}
Heuts constructs an $\infty$-category $\Opnu_{\leq k}$ of non-unital $k$-restricted $\infty$-operads, and proves that the natural restriction functor $\Opnu \to \Opnu_{\leq k}$ has fully faithful left and right adjoints.

In this note, we study the unital version of $k$-restricted $\infty$-operads.
Recall that an $\infty$-operad $\cO$ is \emph{unital} if for every color $X$ the space $\Mul_{\cO}(\varnothing, X)$ is contractible. 
We model $k$-restricted $\infty$-operads using Moerdjik and Weiss' category $\Omega$ of dendroidal trees \cite{MR3100887}. Dendroidal trees are certain trees that parametrize composable operadic operations, just as the $n$-th simplex $[n]$ parametrizes $n$-composable morphisms.

By the results of \cite{Barwick_2018} 
and \cite{two-models}, the $\infty$-category $\Op$ of $\infty$-operads is equivalent to the $\infty$-category $\Segcpl(\Omega)$ of complete Segal presheaves on the dendroidal category $\Omega$. 
Using this, we prove in \cref{cor:opun-eq-Omegac} that the $\infty$-category $\Opun$ of unital $\infty$-operads is equivalent to the $\infty$-category $\Segcpl(\Omegac)$ of complete Segal presheaves on the category $\Omegac$ of \emph{closed} dendroidal trees (\cref{def:closed-tree}).\footnote{This result was first proven in \cite{MR4222648} using model categories on dendroidal sets. We prove this directly.}

Inspired by this, we say that a dendroidal tree is $k$-dendroidal if it is built from $n$-corollas with $n \leq k$, and define the $\infty$-category $\Opunleq{k}$ of unital $k$-restricted $\infty$-categories as complete Segal presheaves on the category $\Omegacleqk$ of closed $k$-dendroidal trees (see \cref{def:k-restristed-unital-operad}). 
When $k = \infty$ we take $\Opunleq{k}$ to be $\Opun$.

Our main result is the following:
\begin{theorem}[\cref{thm:left-adjoint-theorem}, \cref{thm:right-adjoint-theorem}]
    Given $1 \leq k \leq j \leq \infty$, the natural restriction functor 
    $(-)^k \colon \Opunleq{j} \to \Opunleq{k}$ admits a fully faithful left adjoint $\Lk$ as well as a fully faithful right adjoint $\Rk$, given by left and right Kan extension along 
    $i_k \colon \Omegacleqk \hookrightarrow \Omegacleq{j}$ respectively.
\end{theorem}
Intuitively, for a unital $k$-restricted $\infty$-operad $\cO$, the $n$-ary morphism space of $\Lk\cO$ is the space of all possible $n$-ary morphism that can be created from $(\leqk)$-ary morphisms in $\cO$ (\cref{cor:left-adj-description-on-F}); while the $n$-ary morphism space of $\Rk\cO$ is the space of families of $(\leqk)$-ary morphisms that are compatible under taking units (\cref{cor:right-adj-description-on-F}).

To prove the main theorem, we show in \cref{prop:LK-complete-Segal} and \cref{prop:Rk-complete-Segal} that the left and right Kan extensions both preserve complete Segal presheaves.
Furthermore, we recognize the images of $\Rk$ and $\Lk$ in \cref{cor:left-adj-image-of-Lk} and \cref{cor:right-adj-image-of-Rk}.

Given a unital $\infty$-operad $\cO$, we show in \cref{cor:left-fil-colimits} that there exists a converging filtration:
\begin{equation}
    \mathrm{L}_1\cO \to \mathrm{L}_2\cO \to \cdots \to \Lk\cO \to \cdots \to \cO
\end{equation}
where $\Lk\cO \coloneqq \Lk(\cO)^k$. For the unital associative operad $\AA_\infty = \EE_1$, the $\Lk$ filtration is the unital version of \eqref{eq:non-unital-AAn-adjoint} (see \cref{ex:Ak-operads}).
More generally, in \cref{ex:Ek-filtration} we use the results of \cite{goppl} to identify the $\Lk$ filtration of the little cube operads $\EE_n$ with the stratified filtration of the Fulton-MacPherson operads.

Given a unital $\infty$-operad $\cO$, we show in \cref{cor:right-co-fil} that there exists a converging co-filtration:
\begin{equation}
    \cO \to  \cdots \to \Rk \cO \to \cdots \to \mathrm{R}_2 \cO \to \mathrm{R}_1 \cO
\end{equation}
where $\Rk\cO \coloneqq \Rk(\cO)^k$.
The $\Rk$ co-filtration for the $\AA_\infty$ operad is explicitly calculated in \cref{ex:E1-right-adjoint}.

\textbf{Outlook:}
We believe much of our results on unital $k$-restricted $\infty$-operads and the $\Lk$ left adjoint can be extended to general $k$-restricted $\infty$-operads by working with non-necessarily closed $k$-dendroidal trees. In \cite{barkan2023},  Barkan studied 
the left adjoint $\Lk$ for general $\infty$-operads using a notion of $k$-restriction that is closer to Lurie's model. It would be interesting to unify the two different models.

On the other hand, we believe that the $\Rk$ right adjoint does not exist for general $k$-restricted $\infty$-operads. One sign of this is that the underlying $\infty$-category functor $\Op \to \cat$ does not have a right adjoint. The problem is that the right adjoint needs unital data, which varies for general $\infty$-operads. However, the right adjoint should exist when we specify the unital data.
In fact, this is why we restrict ourselves to unital $\infty$-operads and why there is also a right adjoint in the non-unital setting, as shown in \cite{MR4320769}.

In \cite{goppl}, Göppl studied the problem of lifting morphisms 
between reduced  $\infty$-operads from their $k$-restrictions, by matching-and-latching along maps $\Lk \cO \to \cO \to \Rk\cO$.\footnote{A unital $\infty$-operad is reduced if its underlying $\infty$-category is equivalent to $\pt$.} It would be interesting to prove an $\infty$-categorical statement about successively lifting from $\Opunleq{k}$ to $\Opunleq{k+1}$ by some matching-and-latching obstruction theory.

Upon finishing this paper, we received communications from the authors of \cite{KK}, in which they independently proved \cref{thm:right-adjoint-theorem}.

\textbf{Outline:}
In \cref{sec:k-restricted-closed-den}, we review dendroidal trees and the subclasses of closed and $k$-dendroidal trees.
In \cref{sec:unital-k-restricted-op}, we prove that unital $\infty$-operads are equivalent to complete Segal presheaves on closed dendroidal trees and define unital $k$-restricted $\infty$-operads.
We prove the left adjoint and the right adjoint parts of our main theorem in \cref{sec:left-adjoint} and \cref{sec:right-adjoint} respectively.

\textbf{Acknowledgement:}
The authors are grateful to Shaul Barkan for helpful conversations and Sophus Valentin Willumsgaard for comments on a previous draft. ASD would like to thank Shachar Carmeli for his guidance and support, most notably during his time in Copenhagen. 
YLL would also like to thank Mike Hopkins for his guidance and encouragement.
YLL gratefully acknowledges the financial support provided by the Simons Collaboration on Global Categorical Symmetries.

\section{Closed $k$-dendroidal trees}\label{sec:k-restricted-closed-den}

\subsection{Dendroidal trees}\label{subsec:dendroidal}
A finite rooted tree $T$ is a finite poset $(T, \leq)$ such that 
\begin{enumerate}
    \item there exists a minimal element,
    \item for any $x \in T$, the sub-poset $T_{\leq x}$ consisting of elements less than $x$ is linearly ordered.
\end{enumerate}
Following \cite{MR3100887}, a dendroidal tree $(T, L)$  is a finite rooted tree $T$ together with a subset $L$ of its maximal elements. We will often refer to a dendroidal tree $(T,L)$ simply as a tree $T$.

We refer to elements of $T$ as \emph{edges}, the 
minimal edge as the \emph{root}, and elements in $L$ as \emph{leaves}. 
An edge is \emph{external} if it is the root or a leaf; else, we call it \emph{internal}.
For an edge $e \in T$, the \emph{valence} of $e$ is the number of minimal elements in $T_{> e}$. 

\begin{example}\label{ex:edge}
    The \emph{edge} $\eta$ is the dendroidal tree with only one edge. Explicitly, $\eta = \{r\}$ with its unique poset structure and $L = \{r\}$.
\end{example}

\begin{example}\label{ex:k-corolla}
   Fix $n \geq 1$. The \emph{$n$-corolla} $C_n$ is the dendroidal tree with $k$ leaves and no internal edges. The poset underlying $C_k$ is $\{0, \cdots, k\}$ 
   such that $0$ is minimal $L = \{1, \cdots, k\}$. For $n=0$, we take $C_0$ to be $(\{0\}, \varnothing)$.
\end{example}
Every dendroidal tree is obtained from gluing $n$-corollas along their roots and leaves.

\begin{example}\label{ex:linear-posets}
Fix $n \geq 0$. Let $[n]$ be the linear poset $\{0 \leq 1 \leq \cdots \leq n\}$. The pair $([n], \{n\})$ is a dendroidal tree. Note that $[0]$ is the edge $\eta$, and $[1]$ is the $1$-corolla $C_1$.
\end{example}
Let $\Op_\infty$ denote the $\infty$-category of $\infty$-operads, and $\Opone$ be the full subcategory of ordinary operads.
Following \cite{MR3100887}, for each dendroidal tree $T$, there is an ordinary operad $\Free_{\Op}(T) \in \Opone$ whose colors are edges of $T$ and whose operations are generated by non-leaf edges of $T$.
\begin{definition}\label{def:Omega}
Let  $\Omega$ be the category of dendroidal trees whose objects are trees and morphisms are given by 
\begin{equation}
  \Hom_{\Omega}(T, T') \coloneqq   \Hom_{\Opone}(\Free_{\Op}(T), \Free_{\Op}(T')).
\end{equation}
\end{definition}

\begin{remark}
    Let $T$ be a dendroidal tree. Then a morphism from $\eta$ to $T$ corresponds to an edge of $T$.
\end{remark}

\begin{observation}\label{obs:simplex-embeds-in-dendroidal}
Let $\Delta$ be the simplex category.
    There exists a fully faithful inclusion 
    $r \colon \Delta \to \Omega$ taking the object $[n]$ to the corresponding dendroidal tree $([n], \{n\})$ defined in \cref{ex:linear-posets}.
\end{observation}

\subsection{Closed dendroidal trees}\label{subsec:closed-dendroidal}
In this subsection, we review the subcategory $\Omegac$ of closed dendroidal trees \cite{MR4222648} and define various factorization systems on $\Omegac$. 
\begin{definition}\label{def:closed-tree}
A tree $T$ is \emph{closed} if the set of leaves $L$ is empty. 
\end{definition}
We denote by $\Omegac$ the full subcategory $\Omega$ consisting of closed trees. 
\begin{observation}\label{obs:closed-trees-description}
    The category $\Omegac$ has an explicit description: its objects are finite rooted trees; furthermore, a morphism from $T$ to $T'$ is an order-preserving map that preserves independence.
    Note that two elements $x, y \in T$ are \emph{independent} if neither $x \leq y$ nor $y \leq x$ holds, and an order-preserving map $f \colon T \to T'$ \emph{preserves independence} if for every pair of independent elements $x, y \in T$, the pair $f(x)$ and $f(y)$ are also independent in $T'$.
\end{observation}

The fully faithful inclusion $j \colon \Omegac \hookrightarrow \Omega$ has a left adjoint $(-)^c$, which takes a tree $(T, L)$ to its closure $(T, \varnothing)$. In particular, we denote the closure of the edge as $\overline{\eta}$ and the closure of the $n$-corolla as $\CC{n}$. Note that $\overline{\eta} = C_0 = \CC{0}$. Let $\mu$ denote the unique map $\eta \to \CC{0}$. 
We have a nice characterization for $(-)^c$:
\begin{proposition}\label{prop:tree-closure-uni-prop}
   The left adjoint $(-)^c$ exhibits $\Omegac$ as the localization $\Omega[\{\mu\}^{-1}]$.\footnote{See \cite[\S5.2.7]{HTT} for the theory of localization.} That is, $\Omegac$ is the full subcategory of $\mu$-local objects, and for any $T \in \Omegac$ the unit map $T \to T^c$ is a localization relative to $\mu$.
\end{proposition}

\begin{proof}
    let $T$ be a tree, then an edge $e \colon \eta \to T$ can extends to a map $\overline{e} \colon \CC{0} \to T$ if and only if there are no leaves above $e$. Therefore, $\mu$-local objects are precisely the closed trees. Let $T'$ be a closed tree, then the unit map $T \to T^c$ induces an equivalence 
    \begin{equation}
        \Hom_{\Omega}(T^c, T') \to \Hom_{\Omega}(T, T').
    \end{equation}
    Thus, $T \to T'$ is a localization relative to $\mu$.
\end{proof}

Next, we study factorization systems on $\Omegac$.\footnote{We refer the reader to \cite[\S 5.2.8]{HTT} for an introduction to factorization systems.} 
First, we have the various classes of morphisms in $\Omegac$:

\begin{notation}\label{nota:morphisms-in-omegac}
    Let $f \colon T \to T'$ be a morphism in $\Omegac$.
    \begin{enumerate}
        \item $f$ is \emph{rooted} if it takes the root of $T$ to the root of $T'$.
        \item $f$ is called a \emph{subtree inclusion} if $f$ is injective and the image of $f$ is a subtree in $T'$, i.e., for any $e_0, e_2 \in T$ and $e' \in T$ such that $f(e_0) \leq e' \leq f(e_2)$, then there is an $e_1 \in T$ such that $f(e_1) = e'$.
        \item $f$ is called \emph{max-surjective} if for every edge $e' \in T'$ there exists an edge $e \in T$ with $f(e) \geq e'$. Equivalently,  every maximal edge of $T'$ is in the image of $f$. 
   \end{enumerate}
\end{notation}
We will often identify a subtree inclusion $f \colon X \to T$ with its image $f(X) \subset T$.
\begin{remark}
    Rooted max-surjective maps are often called \emph{active}, while subtree inclusions are often called \emph{inert}. We chose our terminology because we will also consider max-surjective maps and rooted subtree inclusions.
\end{remark}
\begin{observation}\label{obs:max-surj}
    Let $f \colon T \to T'$ be a max-surjective map in $\Omegac$. Then $f$ takes maximal elements in $T$ to maximal elements in $T'$.
    Furthermore, since $f$ preserves independence, it restricts to an isomorphism on maximal elements.
\end{observation}

\begin{observation}\label{obs:subtree-data}
    Let $T$ be a closed tree. A rooted subtree inclusion of $T$ corresponds to a subset of pairwise independent elements. A subtree inclusion of $T$ corresponds to a subset $S$ of pairwise independent elements together with an element that is less than all elements in $S$. 
\end{observation}

\begin{proposition}\label{prop:fact-sys-on-Omegac}
    The following holds:
    \begin{enumerate}
        \item\label{item:fact-1} The classes of (max-surjective, rooted subtree inclusion) morphisms define a factorization system on $\Omegac$. 
        \item\label{item:fact-2} The classes of (rooted max-surjective, subtree inclusion) morphisms also define a factorization system on $\Omegac$.
    \end{enumerate}
\end{proposition}

\begin{proof}
    The unique factorization statement for both parts is straightforward.
    To prove part \eqref{item:fact-1}, 
    it remains to show that every map $g \colon T \to T''$ is the composition of a max-surjective map followed by a rooted subtree inclusion.
    Let $T'$ be the rooted subtree of $T''$ consisting of edges $e' \in T''$ such that there exists an edge $e \in T$ with $f(e) \geq e'$.
    The map $g$ factors through $T'$, and by construction, the first map $T \to T'$ is max-surjective.

    Now we turn to part \eqref{item:fact-2}. We would like to show that every map $g \colon T \to T''$ is the composition of a rooted max-surjective map followed by a rooted subtree. In this case, we take the subtree $T'$ of $T''$ to consist of edge $e' \in T''$ such that there exist edges $e_0, e_1 \in T$ with $f(e_0) \geq e' \geq f(e_1)$. The map $g$ factors through $T'$, and by construction, the first map $T \to T'$ is rooted and max-surjective.
\end{proof}

\begin{remark}
    Every morphism $f \colon T \to T'$ in $\Omegac$ factors uniquely as 
\begin{equation}
    T \xrightarrow{f_1} T_1 \xrightarrow{f_2} T_2 \xrightarrow{f_3} T'
\end{equation}
where $f_1$ is rooted max-surjective, $f_2$ is  a  max-surjective subtree inclusion, and $f_2$ is a rooted subtree inclusion.
\end{remark}

\subsection{$k$-dendroidal trees}\label{subsec:k-dendroidal}
Throughout the subsection, let us fix $1 \leq k \leq \infty$.
\begin{definition}\label{def:k-dendroidal}
Let $T$ be a tree. We say that $T$ is a \emph{$k$-dendroidal tree} if every edge of $T$ has valence $\leq k$. We take the definition to be vacuously true when $k = \infty$.
\end{definition}
We denote by $\Omega_{\leqk}$ the full subcategory of $\Omega$ consisting of $k$-dendroidal trees. Similarly, we denote by $\Omegacleqk$ the full subcategory of $\Omegac$ consisting of closed $k$-dendroidal trees. Note that $\Omega_{\leq \infty} = \Omega$ and $\Omegac_{\leq \infty} = \Omegac$.
Given $j > k$, 
We denote by $i_{k}$ the inclusion 
$\Omegacleqk \hookrightarrow \Omegac_{\leq j}$.

Much of the results in \cref{subsec:closed-dendroidal} translate to the setting of $k$-dendroidal trees. The following is an immediate consequence of \cref{prop:tree-closure-uni-prop}:
\begin{corollary}
    The localization functor $(-)^c \colon \Omega \to \Omegac$ restricts to a localization functor $(-)^c \colon \Omega_{\leqk} \to \Omegacleqk$.
\end{corollary}

$k$-dendroidal trees also satisfy crucial closure properties with respect to max-surjective maps and subtree inclusions:
\begin{lemma}\label{lem:k-closure}
    Given a map $f \colon K \to K'$ in $\Omega$. 
    \begin{enumerate}
        \item \label{item-enum:closure-1} If $f$ is a subtree inclusion and $K'$ is a $k$-dendroidal tree, then $K$ is also a $k$-dendroidal tree.
        \item \label{item-enum:closure-2} If $f$ is a max-surjective map and $K$ is a $k$-dendroidal tree, then $K'$ is also a $k$-dendroidal tree.
    \end{enumerate}
\end{lemma}
\begin{proof}
    Part \eqref{item-enum:closure-1} is straightforward. As for part \eqref{item-enum:closure-2}, suppose that $K'$ is not a $k$-dendroidal tree. Then there exists an edge $x \in K'$ with  $(k+1)$ minimal edges $x_1, \cdots, x_{n+1}$ in $K'_{> x}$. As $f$ is max-surjective, there exists edges $e_1, \cdots, e_{n+1}$ in $K$ such that $f(e_i) \geq x_i$ in $K$.

    Let $T_0$ be the sub-poset of edges $e \in T$ satisfying that $e \leq e_i$ for all $i$. $T_0$ is non-empty as the root is in $T_0$, and it is linearly ordered as it is a sub-poset of $T_{< e_1}$, which is linearly ordered. Let $e_0$ be the maximal edge of $T_0$. For each $1 \leq i \leq k+1$, let $e'_{i}$ be the minimal element in the poset $T_{e_0 < - \leq e_i}$.
    For each $i$, we claim that  $f(e'_i) > x$: if not, then $f(e'_i) \leq x$ (as $T'_{< x_i}$ is linearly ordered) which implies that  $e'_i \leq e_j$ for all $j$ as $f$ preserves independence. This contradicts the maximality assumption on $e_0$. It follows that $f(e'_i) \geq x_i$ and the $e'_i$'s are pairwise independent. This implies that the valence of $e_0$ is at least $k+1$, which is a contradiction.
\end{proof}
Lastly, we have the $k$-dendroidal analogue of \cref{prop:fact-sys-on-Omegac}:
\begin{corollary}\label{cor:k-fact-sys}
    The classes of (max-surjective, rooted subtree inclusion) morphisms define a factorization system on $\Omegacleqk$. Similarly, the classes of (rooted max-surjective, subtree inclusion) morphisms define a factorization system on $\Omegacleqk$.
\end{corollary}

\section{Unital $k$-restricted $\infty$-operads}
\label{sec:unital-k-restricted-op}
By the works of \cite{Barwick_2018} 
and \cite{two-models}, the $\infty$-category $\Op$ of $\infty$-operads is equivalent to 
the $\infty$-category 
of complete Segal presheaves on $\Omega$. In this section, we prove an unital version of this statement in \cref{cor:opun-eq-Omegac} and define unital $k$-restricted $\infty$-operads as complete Segal presheaves on $\Omegacleq{k}$.

\subsection{Segal presheaves}
In this subsection, we extend the notion of Segal presheaves to presheaves on closed $k$-dendroidal trees and provide some equivalent yet useful criteria for Segal presheaves.
Let $\Spaces$ denote the $\infty$-category of spaces. For any $\infty$-category $\cC$, we write $\cP(\cC) \coloneqq \Fun(\cC^\op, \Spaces)$ for the $\infty$-category of presheaves on $\cC$. 

\begin{definition}\label{def:Omega-el}
    Let $T$ be a tree. 
    We denote by $(\Omega^{\el})_{/T}$ be the full subcategory of the over-category $\Omega_{/T}$ consisting of morphisms $(f \colon X \to T)$ satisfying the following: 
    \begin{enumerate}
        \item $X$ is either the edge $\eta$ or a corolla $C_n$.
        \item $f$ is a subtree inclusion. Furthermore, if $X$ is $n$-corolla with root $x$, then $f(x)$ also has valence $n$.
    \end{enumerate}
\end{definition}
Note that this assignment is functorial: a subtree inclusion $f \colon T \to T'$ induces a map $f^{\el} \colon (\Omega^{\el})_{/T} \to (\Omega^{\el})_{/T'}$.

\begin{definition}[{\cite[Definition 4.2.1]{MR4038556}}]\label{def:Segal-condition}
    Let $\cF$ be a presheaf on $\Omega$. We say that $\cF$ is a \emph{Segal presheaf} if for every tree $T$, the canonical map 
    \begin{equation}\label{eq:Segal-condition}
        \cF(T) \to \lim_{X \in ((\Omega^{\el})_{/T})^{\op}} \cF(X)
    \end{equation}
    is an equivalence.
\end{definition}

We generalize the notion of Segal presheaves to $\Omegac$ and $\Omegacleqk$:
\begin{definition}\label{def:closed-segal-condition}
    Fix $1 \leq k \leq \infty$ and $\cF$ a presheaf on $\Omegacleqk$. We say that $\cF$ is a \emph{Segal presheaf} if for every closed $k$-dendroidal tree $T$,  the canonical map 
    \begin{equation}\label{eq:closed-segal-condition}
        \cF(T) \to \lim_{X \in ((\Omega^{\el})_{/T})^{\op}} \cF(X^c)
    \end{equation}
    is an equivalence.
\end{definition}
Let $\cC$ be $\Omega$, $\Omegac = \Omegacleq{\infty}$, and $\Omegacleqk$, 
we denote by $\Seg(\cC)$ the full subcategory of Segal presheaves on $\cC$.

\begin{observation}\label{obs:Ck-check-equivalence}
   Let $f \colon \cF_1 \to \cF_2$ be a map of Segal presheaves on $\Omegacleqk$. It follows from the Segal condition that $f$ is an equivalence if and only if the induced map $$f(\CC{i}) \colon \cF_1(\CC{i}) \to \cF_2(\CC{i})$$ is an equivalence for every $0 \leq i \leq k$. 
\end{observation}

\begin{observation}\label{obs:k-restriction-preserves-segal}
Fix $1 \leq k \leq j \leq \infty$.
    It follows directly from the definition that $i_k^* \colon \cP(\Omegac_{\leq j}) \to \cP(\Omegacleqk)$ preserves Segal presheaves, hence restricts to a functor $i_k^* \colon \Seg(\Omegac_{\leq j}) \to \Seg(\Omegacleqk)$. 
\end{observation}
Next we relate Segal presheaves on $\Omega$ and $\Omegac$:
\begin{proposition}\label{prop:uni-prop-for-segal-on-Omegac}
    The following holds:
    \begin{enumerate}
        \item \label{enum-prop:1} 
        The pullback functor $(-)^{c,*} \colon \cP(\Omegac) \to \cP(\Omega)$ preserves Segal presheaves.
        \item \label{enum-prop:2}
         The induced functor
            $(-)^{c, *} \colon \Seg(\Omegac) \to \Seg(\Omega)$
        is fully faithful.
        \item \label{enum-prop:3} The image of $(-)^{c, *}$ consists of Segal presheaves $\cF$ on $\Omega$ such that the map $\cF(\CC{0}) \to \cF(\eta)$ is an equivalence.
    \end{enumerate}
\end{proposition}

\begin{proof}
    Part \eqref{enum-prop:1} is a direct consequence of the definition of Segal presheaves.
    Since  $(-)^c$ is a localization functor  (\cref{prop:tree-closure-uni-prop}), 
    by \cite[Prop.~5.2.7.12]{HTT}, the pullback functor 
    \begin{equation}
        (-)^{c, *} \colon \cP(\Omegac) \to \cP(\Omega)
    \end{equation}
    is fully faithful and its image consists of presheaves $\cF$ on $\Omega$ such that the map $\cF(T^c) \to \cF(T)$ is an equivalence for all $T \in \Omega$. This proves part \eqref{enum-prop:2}.

    Now we prove \eqref{enum-prop:3}. Let $\cF$ be a Segal presheaf $\cF$ on $\Omega$ such that $\cF(\CC{0}) \to \cF(\eta)$ is an equivalence. 
    It follows from the Segal condition that the canonical map $\cF(T^c) \to \cF(T)$ is an equivalence for all $T \in \Omega$, as $T^c$ is obtained from $T$ by gluing $\CC{0}$ to the leaves on $T$. Therefore $\cF$ lifts to a presheaf $\cF' \in \cP(\Omega)$. Furthermore, $\cF'$ is a Segal presheaf as $\cF$ is a Segal presheaf.
\end{proof}

We end this subsection with some useful criteria for the Segal condition.
First, we have to introduce some terminology.
\begin{notation}\label{nota:upper-and-lower-subset}
    Let $T$ be a tree and $e \in T$ be an edge of $T$. 
We denote by $T_{\geq e} = \{x \in T| x \geq e\}$ the \emph{upper subtree} of $T$ with root $e$,  $T^e = \{x \in T|x \ngtr e\}$ the \emph{lower subtree} of $T$ with leaf $e$, and $v(e)$ the subset of minimal elements in $T_{> e}$. When $e$ is a non-maximal edge, we denote by $C_e$ the $|v(e)|$-corolla subtree of $T$ with root $e$.
\end{notation}
Note that $|v(e)|$ is the valence of $e$.
\begin{proposition}\label{prop:cutting-Segal}
   Let $\cF$ be a presheaf on $\Omega$. The following are equivalent:
   \begin{enumerate}
       \item \label{item:cutting-1}
       $\cF$ is a Segal presheaf (\cref{def:Segal-condition}).
       \item \label{item:cutting-2} For every tree $T$ and inner edge $e$ of $T$, the canonical map 
   \begin{equation}\label{eq:cutting-segal-cond-2}
       \cF(T) \to \cF(T_{\geq e}) \times_{\cF(e)} \cF(T^e)  
   \end{equation}
   is an equivalence. 
        \item \label{item:cutting-3} For every tree $T$ and non-maximal edge $e_0$ of $T$, the canonical map 
        \begin{equation}\label{eq:cutting-segal-cond-3}
            \cF(T) \to \cF(T^{e_0}) \times_{\cF(e_0)} \cF(C_e) \times_{\prod_{e \in v(e_0)} \cF(e)} \,\prod_{e \in v(e_0)}\cF(T_{\geq e})
        \end{equation}
        is an equivalence. 
        \item \label{item:cutting-4} For every tree $T$, the canonical map 
        \begin{equation}\label{eq:cutting-segal-cond-4}
            \cF(T) \to \cF(C_e) \times_{\prod_{e \in v(r)} \cF(e)} \,\prod_{e \in v(r)}\cF(T_{\geq e})
        \end{equation}
        is an equivalence.
        Here $r$ is the root of $T$.
   \end{enumerate}
\end{proposition}
\begin{proof}
    Given a tree $T$ and an inner edge $e$ of $T$, consider the following diagram:
    \begin{equation}\label{eq:omega-el-comm}
        \begin{tikzcd}
            (\Omega^{\el})_{/e} \ar[r] \ar[d] & (\Omega^{\el})_{/T^e} \ar[d]\\
            (\Omega^{\el})_{/T_{\geq e}} \ar[r] & (\Omega^{\el})_{/T}.
        \end{tikzcd}
    \end{equation}
    Note that $(\Omega^{\el})_{/e} = \{e\}$. It is straightforward to see that 
    \cref{eq:omega-el-comm} is a pushout of ($\infty$-)categories, that is, 
    $\Omega^{\el}_{/T} \simeq (\Omega^{\el})_{/T_{\geq e}} \sqcup_{\{e\}} (\Omega^{\el})_{/T^e}$.

    Now we show that condition \eqref{item:cutting-1} implies condition \eqref{item:cutting-2}: suppose $\cF$ is a Segal presheaf, then \eqref{eq:cutting-segal-cond-2} is the composite of equivalences: 
    \begin{equation}
        \begin{aligned}
            \cF(T) & \simeq \lim_{X \in ((\Omega^{\el})_{/T})^{\op}} \cF(X)\\ 
            & \simeq \lim_{X \in ((\Omega^{\el})_{/T_{\geq e}} \sqcup_{e} (\Omega^{\el})_{/T^e})^{\op}} \cF(X) \\ 
            & \simeq \left(\lim_{X_1 \in ((\Omega^{\el})_{/T_{\geq e}})^{\op}} \cF(X_1)\right) \times_{\cF(e)} \left(\lim_{X_2 \in ((\Omega^{\el})_{/T^e})^{\op}} \cF(X_2)\right) \\ 
            & \simeq \cF(T_{\geq e}) \times_{\cF(e)} \cF(T^e).
        \end{aligned}
    \end{equation}
    A similar argument shows that condition \eqref{item:cutting-1} implies condition \eqref{item:cutting-3}, which implies condition \eqref{item:cutting-4}.

    Conversely, suppose $\cF$ satisfies condition \eqref{item:cutting-2}, we are going to show that $\cF$ satisfies condition \eqref{item:cutting-1}. 
    We are going to induct on the number $n$ of internal edges of $T$. 
    The base $n = 1$ case is straightforward. For the inductive step, we pick a non-maximal internal edge $e$ of $T$. Then both $T_{\geq e}$ and $T^e$ have fewer than $n$ internal edges. 
    Now the map \eqref{eq:Segal-condition} is an equivalence as it is the composite of equivalences:
    \begin{equation}
        \begin{aligned}
            \cF(T) &\simeq \cF(T_{\geq e}) \times_{\cF(e)} \cF(T^e) \\
            &\simeq \left(\lim_{X_1 \in ((\Omega^{\el})_{/T_{\geq e}})^{\op}} \cF(X_1)\right) \times_{\cF(e)} \left(\lim_{X_2 \in ((\Omega^{\el})_{/T^e})^{\op}} \cF(X_2)\right) \\ 
            &\simeq \lim_{X \in ((\Omega^{\el})_{/T_{\geq e}} \sqcup_{e} (\Omega^{\el})_{/T^e})^{\op}} \cF(X) \\ 
            &\simeq \lim_{X \in ((\Omega^{\el})_{/T})^{\op}} \cF(X).
        \end{aligned}
    \end{equation}
    A similar argument shows that condition \eqref{item:cutting-4} implies condition \eqref{item:cutting-1} by inductively cutting at the root. This completes the proof.
\end{proof}
Now, we move to the closed $k$-dendroidal setting.
For a closed tree $T$ and an internal edge $e$, we denote by $\overline{e}$ and $\overline{T^e}$ the closure of $e$ and $T^e$. Note that $T_{\geq e}$ is already closed. The following proposition can be proven by the same argument:
\begin{proposition}\label{prop:closed-cutting-Segal}
    Fix $1 \leq k \leq \infty$ and  $\cF$ a presheaf on $\Omegac_{\leqk}$. The following are equivalent:
   \begin{enumerate}
       \item \label{item:closed-cutting-1}
       $\cF$ satisfies the Segal condition (\cref{def:closed-segal-condition}).
       \item \label{item:closed-cutting-2} 
       For every closed $k$-dendroidal tree $T$ and inner edge $e$ of $T$,  the canonical map 
   \begin{equation}\label{eq:closed-cutting-segal-cond-2}
       \cF(T) \to \cF(T_{\geq e}) \times_{\cF(\overline{e})} \cF(\overline{T^e})  
   \end{equation}
   is an equivalence. 
        \item \label{item:closed-cutting-3} 
        For every closed $k$-dendroidal tree $T$ and non-maximal edge $e_0$ of $T$, the canonical map 
        \begin{equation}\label{eq:closed-cutting-segal-cond-3}
            \cF(T) \to \cF(\overline{T^{e_0}}) \times_{\cF(\overline{e_0})} \cF(\CC{n}) \times_{\prod_{e \in v(e_0)} \cF(\overline{e})} \,\prod_{e \in v(e_0)}\cF(T_{\geq e})
        \end{equation}
        is an equivalence. 
        \item \label{item:closed-cutting-4} 
        For every closed $k$-dendroidal tree $T$, the canonical map 
        \begin{equation}\label{eq:closed-cutting-segal-cond-4}
            \cF(T) \to \cF(\CC{r}) \times_{\prod_{e \in v(r)} \cF(\overline{e})} \,\prod_{e \in v(r)}\cF(T_{\geq e})
        \end{equation}
        is an equivalence.
        Here $r$ is the root of $T$.
   \end{enumerate}
\end{proposition}

\subsection{Complete Segal presheaves}
In this subsection, we define complete Segal presheaves in the closed $k$-dendroidal setting. We fix $1 \leq k \leq j \leq \infty$.
\begin{notation}
    Recall from \cref{obs:simplex-embeds-in-dendroidal} that we have an inclusion $r \colon \Delta \to \Omega$. We will abuse notation and also denote by $r$ the composite $\Delta \xrightarrow{r} \Omega \xrightarrow{(-)^c} \Omegac$ as well as its factorization through $\Omegacleqk$. 
\end{notation}
\begin{observation}\label{obs:k-1-iso}
    Let $k = 1$. Then, the map $r \colon \Delta \xrightarrow{r} \Omegacleq{1}$ is an equivalence.
\end{observation}
    Pulling back along $r$ takes Segal presheaves on $\Omega$, $\Omegac$, and $\Omegacleqk$ to Segal spaces in the sense of \cite{rezk2001model}. 
\begin{definition}\label{def:complete-segal-presh}
    Let $\cC$ be $\Omega$, $\Omegac$, or $\Omegacleqk$. 
    A Segal presheaf $\cF$ on $\cC$ is \emph{complete} if the Segal space $r^*\cF$ is a \emph{complete} Segal space in the sense of \cite{rezk2001model}.
    We denote by $\Segcpl(\cC)$ the full subcategory of \emph{complete Segal presheaves} on $\cC$.
\end{definition}
By \cite{MR2342834},
    the $\infty$-category $\Segcpl(\Delta)$ of complete Segal spaces is equivalent to the $\infty$-category of $\infty$-categories $\cat$.

\begin{observation}\label{obs:k-1-iso-part-II}
    By definition, the pullback functor 
    $r^* \colon \Seg(\Omegacleqk) \to \Seg(\Delta)$ preserves complete objects.
    For $k = 1$ this induces an isomorphism 
    $\Segcpl(\Omegacleq{1}) \simeq \Segcpl(\Delta) \simeq \cat$.
\end{observation}

\begin{observation}\label{obs:preserves-complete}
    The functors $$i^*_k \colon \Seg(\Omegac_{\leq j}) \to \Seg(\Omegacleqk), \quad 
    (-)^{c,*} \colon \Seg(\Omega^c) \to \Seg(\Omega)$$ both preserve and detect complete objects. 
    We will denote the functors on complete Segal presheaves by the same symbols.
\end{observation}

\subsection{$k$-restricted unital $\infty$-operads}
By the works of \cite{Barwick_2018} 
and \cite{two-models}, the $\infty$-category $\cPseg(\Omega)$
of complete Segal presheaves on $\Omega$ is equivalent to the $\infty$-category $\Op$ of $\infty$-operads:
\begin{theorem}[{\cite[Thm.~1.1]{two-models}, \cite[Thm.~10.16]{Barwick_2018}}]\label{thm:op-equivalent-segcpl}
    We have an equivalence of $\infty$-categories $$\Segcpl(\Omega) \simeq \Op.$$ Under this equivalence, the  edge $\eta$ corresponds to the $\infty$-operad $\Triv$, and the $n$-corolla $C_n$ corresponds to $\Free_{\Op}(C_n)$, viewed as an $\infty$-operad.
    
    Let $\cO$ be an $\infty$-operad and $\cF_{\cO}$ be the corresponding complete Segal presheaf; the space of colors $\underline{O}^\simeq$ is isomorphic to $\cF_{\cO}(\eta)$. Furthermore, given colors $X_1, \ldots, X_n, Y \in \cO$, we have an equivalence
    \begin{equation}
        \Mul_{\cO}(X_1, \ldots, X_n; Y) \simeq \cF_{\cO}(C_k) \times_{\cF_{\cO}(\eta)^{\times (n+1)}}(X_1, \ldots, X_n, Y).
    \end{equation}
    Here we view $(X_1, \ldots, X_n, Y)$ as a point in $\cF_{\cO}(\eta)^{\times (n+1)}$ by the isomorphism $\underline{O}^\simeq \simeq \cF_{\cO}(\eta)$.
\end{theorem}
Now we relate complete Segal presheaves on $\Omegac$ to unital $\infty$-operads.
\begin{definition}\label{def:unital-infty-operad}
   An $\infty$-operad $\cO$ is \emph{unital} if for every color $X \in \cO$, the space $\Mul_{\cO}(\varnothing; X)$ is contractible. 
\end{definition}
We denote by $\Opun$ the full subcategory of $\Op$ consisting of unital $\infty$-operads.
\begin{corollary}\label{cor:opun-eq-Omegac}
    The equivalence $\Op \simeq  \Segcpl(\Omega)$ restricts to an equivalence of full subcategories $\Opun \simeq \Segcpl(\Omegac)$.
\end{corollary}
\begin{proof}
    Under the equivalence in \cref{thm:op-equivalent-segcpl}, an $\infty$-operad $\cO$ is unital if and only if 
$\cF_{\cO}(\CC{0}) \to \cF_{\cO}(\eta)$ is an equivalence. Now, the result follows from \cref{prop:uni-prop-for-segal-on-Omegac}.
\end{proof}

Motivated by \cref{cor:opun-eq-Omegac}, we have the following definition for $k$-restricted unital $\infty$-operad:
\begin{definition}\label{def:k-restristed-unital-operad}
    A $k$-restricted unital $\infty$-operad is a complete Segal presheaf on $\Omegacleqk$.
\end{definition}

\begin{notation}\label{nota:Opunleqk}
    From now on we will use $\Opunleqk$ to denote the $\infty$-category $\Segcpl(\Omegacleqk)$ and $(-)^{k}$ to denote the functor 
    \begin{equation}\label{eq:k-restriction-functor}
    (-)^{k} \colon   \Opunleq{j} = \Segcpl(\Omegacleq{j}) \xrightarrow{i_k^*} \Segcpl(\Omegacleqk) = \Opunleqk.
    \end{equation}
    Let $\cF$ be a unital $k$-restricted $\infty$-operad. We refer to $\cF(\CC{0})$ as its space of colors. Additionally, given $1 \leq n \leq k$ and colors $X_1, \cdots, X_n, Y$ in $\cF$, we denote by $\Mul_{\cF}(X_1, \cdots, X_n; Y)$ the fiber product $\cF(\CC{n}) \times_{\cF(\CC{0})^{\times (n+1)}} \{(X_1, \cdots, X_n, Y)\}$.
\end{notation}
By \cref{obs:Ck-check-equivalence}, we have the following observation:
\begin{observation}\label{obs:k-restricted-op-eq}
    Let  $f \colon \cF_1 \to \cF_2$ be a map of unital $k$-restricted $\infty$-operads. The following are equivalent:
    \begin{enumerate}
        \item $f$ is an equivalence.
        \item For any $0 \leq n \leq k$, the induced map 
        \begin{equation}
            \cF_1(\CC{n}) \to \cF_2(\CC{n})
        \end{equation}
        is an equivalence.
        \item The induced map on colors $$\cF_1(\CC{0}) \to \cF_2(\CC{0})$$ is an equivalence; furthermore, for any $1 \leq n \leq k$ and colors $X_1, \cdots, X_n, Y$ in $\cF_1$, the induced map  
        \begin{equation}
            \Mul_{\cF}(X_1, \cdots, X_n; Y) \to \Mul_{\cF'}(f(X_1), \cdots, f(X_n); f(Y))
        \end{equation}
        is an equivalence.
    \end{enumerate}
\end{observation}
Pulling back  complete Segal presheaves along $r \colon \Delta \to \Omegacleqk$ induces a functor 
\begin{equation}
    r^* \colon \Opunleq{k}  = \Segcpl(\Omegacleqk) \to \Segcpl(\Delta) \simeq \cat.
\end{equation}
It takes a unital $k$-restricted $\infty$-operad 
to its underlying $\infty$-category. 

\begin{example}\label{ex:k-1-part-I}
    Let $k = 1$. By \cref{obs:k-1-iso-part-II}, the functor $r^* \colon \Opunleq{1} \to \cat$ is an equivalence. 
    The inverse takes an $\infty$-category $\cC$ to the $1$-restricted $\infty$-operad whose underlying $\infty$-category is $\cC$ and has a unit to each object of $\cC$.
\end{example}

\section{The left adjoint $\Lk$}\label{sec:left-adjoint}
Given $1 \leq k \leq j \leq \infty$, 
we have an adjunction 
\begin{equation}\label{eq:LKan-k-adjunction}
    \begin{tikzcd}[column sep=1.5cm]
\LKan_k \colon \cP(\Omegacleqk) 
\arrow[hookrightarrow, yshift=0.9ex]{r}{}
\arrow[leftarrow, yshift=-0.9ex]{r}[yshift=-0.15ex]{\bot}[swap]{}
&
\cP(\Omegacleq{j}) \colon i_k^*
\end{tikzcd}
\end{equation}
where the fully faithful left adjoint $\LKan_k$ is the left Kan extension along the inclusion $i_k \colon \Omegacleqk \hookrightarrow \Omegacleq{j}$. In this section, we show that the $\LKan_k$ preserves complete Segal objects, thus restricting to a fully faithful left adjoint
\begin{equation}
    \Lk \colon \Opunleq{k} \to \Opunleq{j}
\end{equation}
to the restriction functor 
$(-)^k \colon \Opunleq{j} \to \Opunleq{k}$ defined in \cref{nota:Opunleqk}.
\subsection{Maps to $k$-dendroidal trees}
In this subsection, we established some technical results needed for the left adjoint statement.
First, we have to introduce some notations:
\begin{notation}\label{nota:over-under-cat}
  Given a functor $F \colon \cC \to \cD$ between $\infty$-categories and $d \in \cD$, we denote by $\cC_{d/}$ the fiber product $\cC \times_{\cD} \cD_{d/}$, where $\cD_{d/}$ is the under-category of $d$. Similarly we denote by $\cC_{/d}$ the fiber product $\cC \times_{\cD} \cD_{/d}$, where $\cD_{/d}$ is the over-category of $d$.
\end{notation}

Recall that a functor of $\infty$-categories $F \colon \cC \to \cD$ is \emph{coinitial} if for every $d \in \cD$ the $\infty$-category $\cC_{/d}$ is weakly contractible.\footnote{A morphism is coinitial if and only if $F^\op \colon \cC^\op \to \cD^\op$ is cofinal in the sense of \cite[\S4.1.1]{HTT}. In \cite{HA}, it is referred to as \emph{right cofinal}.}
Note that left adjoints are coinitial: 
if $F$ is a left adjoint, then for every $d$ in $\cD$ the $\infty$-category $\cC_{/d}$ is weakly contractible as it has a final object.

\begin{definition}\label{def:Omegaleqkrmax}
    Given a closed tree $T$ and $k \geq 1$.
    Let $(\Omegacleqk)_{T/^{\rmax}}$ denote the full subcategory of $(\Omegacleqk)_{T/}$ whose objects are rooted max-surjective morphisms (see \cref{nota:morphisms-in-omegac})  from $T$ to closed $k$-dendroidal trees.
\end{definition}
For $k = \infty$ we simply denote $(\Omegacleq{\infty})_{T/^{\rmax}}$ as $(\Omegac)_{T/^{\rmax}}$.
\begin{lemma}\label{lem:left-adjoint-1}
    Given $k \geq 1$ and a closed tree $T$.
    We have an adjunction 
    \begin{equation}\label{eq:j-to-k-restricted-rmax-adjunction}
    \begin{tikzcd}[column sep=1.5cm]
(\Omegacleqk)_{T/^{\rmax}}
\arrow[hookrightarrow, yshift=0.9ex]{r}{}
\arrow[leftarrow, yshift=-0.9ex]{r}[yshift=-0.15ex]{\bot}[swap]{}
&
(\Omegacleqk)_{T/} \colon \Fact(-),
\end{tikzcd}
    \end{equation}
where the left adjoint is the canonical inclusion, and the right adjoint is given by taking the factorization with 
    respect to the (rooted max-surjective, subtree inclusion) factorization system on $\Omegac$ constructed in \cref{prop:fact-sys-on-Omegac}\eqref{item:fact-1}.
\end{lemma}
\begin{proof}
    The (rooted max-surjective, subtree inclusion) factorization system on $\Omegac$ gives an adjunction 
    \begin{equation}
      \begin{tikzcd}[column sep=1.5cm]
        (\Omegac)_{T/^{\rmax}}
        \arrow[hookrightarrow, yshift=0.9ex]{r}{}
        \arrow[leftarrow, yshift=-0.9ex]{r}[yshift=-0.15ex]{\bot}[swap]{}
        &
        (\Omegac)_{T/} \colon \Fact(-).
        \end{tikzcd}  
    \end{equation}
    By \cref{lem:k-closure}\eqref{item-enum:closure-1}, this restricts to the desired adjunction.
\end{proof}

Suppose we have a map of closed trees $f \colon T \to T'$. Consider the composite 
\begin{equation}
    f^{\rmax} \coloneqq 
(\Omegacleqk)_{T'/^{\rmax}}
\hookrightarrow
(\Omegacleqk)_{T'/}
\xrightarrow{f^*} 
(\Omegacleqk)_{T/}
\xrightarrow{\Fact(-)}
(\Omegacleqk)_{T/^{\rmax}}.
\end{equation}
Explicitly, this takes a rooted max-surjective map $(g' \colon T' \to X')$ to 
$(\Fact(g' \circ f) \colon T \to X)$, where $X$ is the (rooted max-surjective, subtree inclusion) factorization of the composite $T \xrightarrow{f} T \xrightarrow{g'} X'$.

Now we study how $(\Omegacleqk)_{T/^{\rmax}}$ behave with respect to cutting a tree along an edge $e$. Fix a closed tree $T$ and $e$ an internal edge of $T$. We have closed subtrees $T_{\geq e}$ and $\overline{T^e}$ from \cref{nota:upper-and-lower-subset}.
Let $\inc_e \colon T_{\geq e} \to T$ and $\inc^e \colon \overline{T^e} \to T$ be the inclusion maps.
Suppose we have a max-surjective $g \colon T \to X$ taking $e$ to $x$ in $X$. Then $\inc_e(g \colon T \to X)$ is $(g_e \colon T_{\geq e} \to X_{\geq x})$ and $\inc^e(g \colon T \to X)$ is $(g^e \colon \overline{T^e} \to \overline{X^x})$. 

This invites an inverse construction:
given two rooted max-surjective maps $g_e \colon T_{\geq e} \to X_1$ and $g^e \colon T^e \to X_2$, we can build a tree
$X_1 \sqcup_{e} X_2$ by gluing the root of $X_1$ to the outer edge $x \coloneqq g^e(e)$ of $X_2$. Furthermore, the maps $g_e$ and $g^e$ induce a rooted max-surjective map $g \colon T \to X_1 \sqcup_e X_2$. 
Note that $X_1 \sqcup_{e} X_2$ is $k$-dendroidal if  $X_1$ and $X_2$ are $k$-dendroidal.
This defines a functor 
\begin{equation}
    \glue_e \colon (\Omegacleqk)_{T_{\geq e}/^{\rmax}} \times (\Omegacleqk)_{\overline{T^e}/^{\rmax}} \to  (\Omegacleqk)_{T/^{\rmax}}.
\end{equation}

The following lemma is clear from the construction:
\begin{lemma}\label{lem:left-kan-cutting-lem}
    Given $k \geq 1$, a closed tree $T$ and an internal edge $e$ of $T$. The functor
    \begin{equation}
        (\inc_e^{\rmax}, \inc^{e,\rmax}) \colon 
        (\Omegacleqk)_{T/^{\rmax}} \to (\Omegacleqk)_{T_{\geq e}/^{\rmax}} \times (\Omegacleqk)_{\overline{T^e}/^{\rmax}}
    \end{equation}
    is an equivalence with the inverse given by $\glue_e$.
\end{lemma}
\subsection{Left adjoint $\Lk$}
Throughout this subsection, we fix $1 \leq k \leq j \leq \infty$.
\begin{proposition}\label{prop:LK-complete-Segal}
    The left Kan extension $\LKan_k \colon \cP(\Omegacleqk) \to \cP(\Omegacleq{j})$ preserves complete Segal presheaves.
\end{proposition}

\begin{proof}
    We first prove that $\LKan_k$ preserves Segal presheaves.
    Given a Segal presheaf $\cF \in \Seg(\Omegacleqk)$.
    By \cref{prop:closed-cutting-Segal}\eqref{item:closed-cutting-2}, it suffices to show the following:
    for any closed $j$-dendroidal tree $T$ and $e$ an internal edge of $T$,
    the canonical map
     \begin{equation}\label{eq:Lkan-desired-eq}
        \LKan_k\cF(T) \to \LKan_k\cF(T_{\geq e}) \times_{\LKan_k\cF(\overline{e})}  \LKan_k\cF(\overline{T^e})
    \end{equation}
    is an equivalence. 

    By \cref{lem:left-adjoint-1}, we see that the inclusion 
    $(\Omegacleqk)_{T/^{\rmax}} \hookrightarrow (\Omegacleqk)_{T/}$ is coinitial. Thus, the induced map
    \begin{equation}
    \colim_{X \in ((\Omegacleqk)_{T/^{\maxsurj}})^\op} \cF(X) \to   \colim_{X \in ((\Omegacleqk)_{T/})^\op} \cF(X) = \LKan_k\cF(T)
    \end{equation}
    is an equivalence. For brevity, we will omit the $\op$.

    Now we unpack \eqref{eq:Lkan-desired-eq} as a sequence of equivalences.
    \begin{equation}
        \begin{aligned}
            \LKan_k\cF(T) &\simeq  \colim_{X \in (\Omegacleqk)_{T/^{\maxsurj}}}\cF(X)\\ 
       &\simeq \colim_{(X_1, X_2) \in  (\Omegacleqk)_{T_{\geq e}/^{\rmax}} \times (\Omegacleqk)_{\overline{T^e}/^{\rmax}}} \cF(X_1 \sqcup_e X_2) \\[5pt]  
            &\simeq \colim_{(X_1, X_2) \in  (\Omegacleqk)_{T_{\geq e}/^{\rmax}} \times (\Omegacleqk)_{\overline{T^e}/^{\rmax}}} \cF(X_1) \times_{\cF({\overline{e}})} \cF(X_2) \\[5pt] 
            &\simeq\colim_{X_1 \in (\Omegacleqk)_{T_{\geq e}/^{\maxsurj}}} \left(
            \colim_{X_2 \in (\Omegacleqk)_{\overline{T^e}/^{\maxsurj}}} \cF(X_1) \times_{\cF({\overline{e}})} \cF(X_2)
            \right)
            \\[5pt] 
            &\simeq \colim_{X_1 \in (\Omegacleqk)_{T_{\geq e}/^{\maxsurj}}} \left(\cF(X_1) \times_{\cF({\overline{e}})} 
                \colim_{X_2 \in (\Omegacleqk)_{\overline{T^e}/^{\maxsurj}}} \cF(X_2)
                \right) \\[5pt]  
            &\simeq  \left(\colim_{X_1 \in (\Omegacleqk)_{T_{\geq e}/^{\maxsurj}}} \cF(X_1)\right)
             \times_{\cF({\overline{e}})} 
            \left(\colim_{X_2 \in (\Omegacleqk)_{\overline{T^e}/^{\maxsurj}}} \cF(X_2)\right) \\[5pt]  
            &\simeq \LKan_k\cF(T_{\geq e}) \times_{\LKan_k \cF(\bar{e})}
            \LKan_k\cF(\overline{T^e}).
        \end{aligned}
    \end{equation}
    We used \cref{lem:left-kan-cutting-lem} for the second equivalence and \cref{prop:closed-cutting-Segal}\eqref{item:cutting-2} for the third equivalence. Additionally, we used the fact that colimits are universal in $\Spaces$ (see \cite[\S 6.1.1]{HTT}) to commute colimits with fiber products.

    Now it remains to show $\LKan_k$ preserves complete Segal presheaves. This follows from the fact that the map $r \colon \Delta \to \Omegacleq{j}$ factors through $\Omegacleqk$ and that $i_k^*( \LKan_k \cF) \simeq \cF$ for any presheaf $\cF$ on $\Omegacleqk$.
\end{proof}

\begin{definition}\label{def:Lk-def}
    Let $\Lk$ denote the functor
    \begin{equation}
        \Opunleqk = \Segcpl(\Omegacleqk) \xhookrightarrow{\LKan_k} \Segcpl(\Omegacleq{j}) = \Opunleq{j}.
    \end{equation}
\end{definition}
By \cref{prop:LK-complete-Segal}, we have the left adjoint version of our main statement:
\begin{theorem}\label{thm:left-adjoint-theorem}
    The adjunction \eqref{eq:LKan-k-adjunction}
    restricts to an adjunction 
    \begin{equation}\label{eq:Lk-adjunction}
    \begin{tikzcd}[column sep=1.5cm]
\Lk \colon \Opunleqk
\arrow[hookrightarrow, yshift=0.9ex]{r}{}
\arrow[leftarrow, yshift=-0.9ex]{r}[yshift=-0.15ex]{\bot}[swap]{}
&
\Opunleq{j} \colon (-)^k
\end{tikzcd}
    \end{equation}
    Furthermore, the left adjoint $\Lk$ is fully faithful. 
\end{theorem}

We have a colimit description of the multi-ary spaces of $\cF$: 
\begin{corollary}\label{cor:left-adj-description-on-F}
    Let $\cF$ be a unital $k$-restricted $\infty$-operad. Then 
    $\Lk \cF(\CC{n})$ is the colimit 
    \begin{equation}
        \colim_{X \in ((\Omegacleqk)_{\CC{n}/^{\maxsurj}})^\op} \cF(X).
    \end{equation}
\end{corollary}
Intuitively, the space of $n$-ary morphisms of $\Lk \cF$ is the space of all possible $n$-ary morphisms that can be created from $\leq k$-ary morphisms in $\cF$. 

By \cref{obs:k-restricted-op-eq}, we can detect when a unital $k$-restricted $\infty$-operad $\cF$ is in the image of $\Lk$:
\begin{corollary}\label{cor:left-adj-image-of-Lk}
    A unital $j$-restricted $\infty$-operad $\cF$ is in the image of $\Lk$ if and only if for each $k <n \leq j$ the canonical map 
    \begin{equation}
        \colim_{X \in ((\Omegacleqk)_{\CC{n}/^{\maxsurj}})^\op} \cF(X) \to \cF(\CC{n})
    \end{equation}
    is an equivalence.
\end{corollary}

\begin{example}\label{ex:k-1-left-adjoint}
    Let $k = 1$. By \cref{ex:k-1-part-I} we see that $\Opunleq{1} \simeq \cat$. The fully faithful composite 
    \begin{equation}
        \cat \simeq \Opunleqk \xrightarrow{\mathrm{L}_1} \Op
    \end{equation}
    takes an $\infty$-category $\cC$ to the unital $\infty$-operad with underlying $\infty$-category $\cC$ and no $n$-ary morphisms for $n \geq 2$. This is an unital analogue of the embedding $\cat \to \Op$ that takes an $\infty$-category to itself viewed as a $\infty$-operad with only $1$-ary morphisms.
\end{example}

Let us write the colocalization functor $\Lk \circ (-)^k$ simply as $\Lk$. 
Let $\cO$ be a unital $\infty$-operad, we have a filtration
\begin{equation}\label{eq:left-adjoint-fil}
    \mathrm{L}_1\cO \to \mathrm{L}_2\cO \to \cdots \to \Lk\cO \to \cdots \to \cO.
\end{equation}
The $k$-th stage $\Lk\cO$ is the $k$-th arity approximation of $\cO$, as they agree on $(\leq k)$-ary morphisms. This implies the following:
\begin{corollary}\label{cor:left-fil-colimits}
   Let $\cO$ be a unital $\infty$-operad. The canonical map  
   \begin{equation}
    \colim_{k \in \mathbb{N}_{\geq 1}} \Lk\cO \to \cO
   \end{equation} 
   is an equivalence.
\end{corollary}

Let us extend the discussion to the little cube operads:
\begin{example}\label{ex:Ak-operads}
    Fix $k \geq 1$.
    Let $\AA_\infty$ be the unital associative operad, then $\AA_k \coloneqq \Lk\AA_\infty$ is the $\infty$-operad generated by $\leq k$-arity morphisms in $\AA_\infty$. As expected, we get a converging filtration
    \begin{equation}
        \AA_1 \to \AA_2 \to \cdots \to \AA_k \to \cdots \to \AA_\infty.
    \end{equation}
    Indeed, $\AA_k$ is the unital analogue of $\AAnu_k$ we encountered in the introduction.
    Let us explicitly compute the $n$-ary morphism space of $\AA_k(n)$.
    Following \cite[Example 2.1.6]{goppl}, there exists a topological model for $\AA_\infty$ whose $n$-ary space is $K_n \times S_n$. Recall that $K_n$ is the $n$-th Stasheff associahedron, which is a convex (hence contractible) $(n-2)$-dimensional polytope.
    
    In \cite[Example 3.1.13]{goppl}, Goppl shows that the space $\AA_k(k+1)$ can be represented by $\partial K_{k+1}$, and the map $\AA_{k}(k+1) \to \AA_\infty(k+1)$ is the canonical inclusion.
    Goppl's argument can be extended to show that 
    for any $n > k$, the space $\AA_k(n)$ can be represented by $\partial_{k-2} K_n$, which is the union of $\leq (k-2)$-faces of $K_n$, and the map $\AA_k(n)
    \to \AA_{\infty}(n)$ is the canonical inclusion
    \begin{equation}
        \partial_{k-2} K_n \times S_n \hookrightarrow K_{n} \times S_n.
    \end{equation} 
\end{example}

\begin{example}\label{ex:Ek-filtration}
    Let $\EE_n$ be the little cube $\infty$-operad (see \cite[Definition 5.1.0.4]{HA}). Following \cite[Example 2.1.6]{goppl}, for any $k \geq 1$ there exists a topological operad called the \emph{Fulton-MacPherson operad} $\FM_n$ 
     representing $\EE_n$. For each $m \geq 1$, the $n$-ary space $\FM_n(m)$ is a compact topological manifold with corners. The manifold $\FM_n(m)$ is naturally stratified over the poset $\Psi(m)$ of closed trees with $m$ labeled maximal edges and no edges of valence $1$. For $k < m$, let $\FM_n^{\leq k}(m)$ be the closed subspace of $\FM_n(m)$ lying over the sub-poset $\Psi_{\leq k}(m)$ of $\Psi(m)$ consisting of $k$-dendroidal trees.
     Generalizing \cite[Example 3.1.13]{goppl}, we can identify the $m$-ary morphism space of $\Lk\EE_n$ with $\FM_n^{\leq k}(m)$, and the map $\Lk\EE_n(m) \to \EE_n(m)$ with the canonical inclusion
     \begin{equation}
         \FM_n^{\leq k}(m) \hookrightarrow \FM_n(m).
     \end{equation}
\end{example}

\section{The right adjoint $\Rk$}\label{sec:right-adjoint}
Now we turn to the right adjoint. Given $1 \leq k \leq j \leq \infty$, analogous to \eqref{eq:LKan-k-adjunction}, we have an adjunction
\begin{equation}\label{eq:RKan-k-adjunction}
    \begin{tikzcd}[column sep=1.5cm]
i_k^* \colon \cP(\Omegacleq{j})
\arrow[ yshift=0.9ex]{r}{}
\arrow[hookleftarrow,yshift=-0.9ex]{r}[yshift=-0.15ex]{\bot}[swap]{}
&
 \cP(\Omegacleqk) \colon \RKan_k 
\end{tikzcd}
\end{equation}
where the fully faithful right adjoint $\RKan_k$ is the right Kan extension along the inclusion $i_k \colon \Omegacleqk \hookrightarrow \Omegacleq{j}$. In this section, we show that the $\RKan_k$ preserves complete Segal objects, thus restricting to a fully faithful right adjoint
\begin{equation}
    \Rk \colon \Opunleq{k} \to \Opunleq{j}
\end{equation}
to the restriction functor 
$(-)^k \colon \Opunleq{j} \to \Opunleq{k}$ defined in \cref{nota:Opunleqk}.

\subsection{Rooted subtrees}
In this subsection, we establish some technical results needed for the right adjoint statement.
Recall that a functor of $\infty$-categories $G \colon \cC \to \cD$ is \emph{cofinal} (\cite[\S4.1.1]{HTT}) if for every $d \in \cD$ the $\infty$-category $\cC_{d/}$ is weakly contractible.
Note that right adjoints are cofinal: 
if $G$ is a right adjoint, then for every $d$ the $\infty$-category $\cC_{d/}$ is weakly contractible as it has an initial object.
\begin{definition}\label{def:Omegaleqkrsub}
    Given $k \geq 1$ and a closed tree $T$, we denote by  $(\Omegacleqk)_\rsubT$ the full subcategory of $(\Omegacleqk)_{/T}$ whose objects are rooted subtree inclusions from closed $k$-dendroidal trees to $T$. 
\end{definition}

When $k = \infty$ we simply denote $(\Omegacleq{\infty})_\rsubT$ as $(\Omegac)_{\rsubT}$. For a set $S$, we denote by $\cP(S)$ the power set of $S$, which we viewed as a category with morphisms being inclusions. We denote by $\cP_{\leq k}(S)$ the full subcategory of $\cP(S)$ consisting of subsets $I \subset S$ with $|I| \leq k$.
\begin{example}\label{ex:Cn-rsub}
   Let $T$ be the closed $n$-corolla $\CC{n}$ with root $r$.
   By \cref{obs:subtree-data}, the category
    $(\Omegac)_{^\rsub/\CC{n}}$
   is equivalent to the power set $\cP(v(r))$.
   For $k \geq 1$, the full subcategory $(\Omegacleqk)_{^\rsub/\CC{n}}$ corresponds to  $\cP_{\leq k}(v(r))$.
\end{example}

\begin{lemma}\label{lem:right-adjoint-1}
    Given $k \geq 1$ and a closed tree $T$, we have an adjunction 
\begin{equation}\label{eq:j-to-k-restricted-rsub-adjunction}
    \begin{tikzcd}[column sep=1.5cm]
\Fact(-) \colon (\Omegacleqk)_{/T}
\arrow[yshift=0.9ex]{r}{}
\arrow[hookleftarrow, yshift=-0.9ex]{r}[yshift=-0.15ex]{\bot}[swap]{}
&
 (\Omegacleqk)_\rsubT
\end{tikzcd}
    \end{equation}
    where the right adjoint is the canonical inclusion, and the left adjoint is given by taking the factorization with 
    respect to the (max-surjective, rooted subtree inclusion) factorization system on $\Omegac$ constructed in \cref{prop:fact-sys-on-Omegac}\eqref{item:fact-2}.
\end{lemma}

\begin{proof}
    The (max-surjective, rooted subtree inclusion) factorization system on $\Omegac$ gives an adjunction 
    \begin{equation}
      \begin{tikzcd}[column sep=1.5cm]
        \Fact(-) \colon (\Omegac)_{/T}
        \arrow[yshift=0.9ex]{r}{}
        \arrow[hookleftarrow, yshift=-0.9ex]{r}[yshift=-0.15ex]{\bot}[swap]{}
        & (\Omegac)_{\rsubT}
        \end{tikzcd}  
    \end{equation}
    By \cref{lem:k-closure}\eqref{item-enum:closure-2}, this restricts to the desired adjunction \eqref{eq:j-to-k-restricted-rmax-adjunction}.
\end{proof}

Given a closed tree $T$ with root $r$. By \cref{nota:upper-and-lower-subset} we have $C_r$ the $|v(r)|$-corolla subtree of $T$ with root $r$. We get a functor 
\begin{equation}
   \phi \colon (\Omegac)_\rsubT \to (\Omegac)_{^\rsub/\CC{r}}
\end{equation}
given by taking a rooted subtree $X$ of $T$ to $X\cap \CC{r}$. Note that the subtree $X \cap \CC{r}$ is non-empty as it contains the root $r$.
By \cref{ex:Cn-rsub} we get a composition
\begin{equation}\label{eq:rsub-eq}
    p \colon (\Omegac)_\rsubT \xrightarrow{\phi} (\Omegac)_{^\rsub/\CC{r}} \simeq \cP(v(r)).
\end{equation}
This takes a rooted subtree $X$ of $T$ to the subset $I$ of $v(r)$ consisting of leaves of $C_r$ that are in $X$. For $k \geq 1$, 
the functor  $p$ restricts to a functor 
\begin{equation}\label{eq:rsub-eq-k}
    p_k \colon (\Omegacleqk)_\rsubT \xrightarrow{\phi}  \cP_{\leqk}(v(r)).
\end{equation}
We will view $(\Omegac)_\rsubT$ and $(\Omegacleqk)_\rsubT$ as categories over $\cP(v(r))$ and $\cP_{\leqk}(v(r))$ respectively.

    \begin{definition}\label{def:right-adjoint-new-cat}
    Given a closed tree $T$  with root $r$. 
    We define a poset $\Omegacrsub(T)$ as follows:
    \begin{enumerate}
        \item An object of $\Omegacrsub(T)$ is a pair $(I, X_I)$ where $I$ is a subset of $v(r)$ and $X_I$ is a collection of rooted subtrees $X_e$ of $T_{\geq e}$, one for each $e \in I$.
        \item Given two objects $(I, X_I)$ and $(I', X'_{I'})$, we have $(I, X_I) \leq (I', X'_{I'})$ if $I \subset I'$ and $X_e \subset X'_{e}$ for every $e \in I$.
    \end{enumerate}
\end{definition}
We view $\Omegacrsub(T)$ as a category with morphisms being $\leq$.
There is a canonical projection functor 
\begin{equation}
    q \colon \Omegacrsub(T) \to \cP(v(r))
\end{equation}
 given by taking $(I, X_I)$ to $I$.
\begin{lemma}\label{lem:q-coCart}
    The functor $q \colon \Omegacrsub(T) \to \cP(v(r))$ 
    is a coCartesian fibration.\footnote{We refer the reader to \cite[\S 2.4]{HTT} for the theory of coCartesian fibrations.}
\end{lemma}
\begin{proof}
 Given an object $(I, X_I)$ in $\Omegacrsub(T)$ and an inclusion $I \subset I'$, we can define a new object $(I', X'_{I'})$ where  
    \begin{equation}
        X'_{e'} = \begin{cases}
            X_e & e \in I \\ 
            \{e'\} & e \neq I.
        \end{cases}
    \end{equation}
    The canonical map $(I, X) \to (I', X'_{I'})$ is a $q$-coCartesian morphism follows from the fact that the root $\{e'\}$ of $T_{\geq e'}$ is the minimal object in the poset of rooted subtrees of $T_{\geq e'}$.
\end{proof}
Given a closed tree $T$ with root $r$, we are going to show that $\Omegacrsub(T)$ and $(\Omegac)_\rsubT$ are isomorphic as categories over $\cP(v(r))$.
Given an object $(I, X_I)$ in $\Omegacrsub(T)$. 
Let $\CC{I}$ be the sub-corolla of $\CC{r}$ consisting of edges in $I$; we have a rooted subtree $\CC{I} \sqcup_{e \in I}X_e$ of $T$ given by gluing the root of $X_e$ to the maximal edge $e$ of $\CC{I}$ for each $e \in I$.

The assignment $(I,X) \mapsto \CC{I} \sqcup_{e \in I}X_e$ defines a functor 
\begin{equation}
    \Gamma(T) \colon \Omegacrsub(T) \to (\Omegac)_\rsubT
\end{equation}
over $\cP(v(r))$.

\begin{proposition}\label{prop:Gamma-equivalence}
    Given a closed tree $T$, the functor $\Gamma(T)$ is an equivalence.
\end{proposition}
\begin{proof}
    Let us construct the inverse functor. Given $X$ a rooted subtree of $T$. Let $I$ be the subset of $v(e)$ consisting of edges in $X$.
    The assignment 
    $X \mapsto (I, \{X_{\geq e}\}_{e \in I})$ defines a functor 
    \begin{equation}
      \psi(T) \colon  (\Omegac)_\rsubT \to \Omegacrsub(T)
    \end{equation}
    over $\cP(v(r))$. It is straightforward to check that $\Gamma(T)$ and $\psi(T)$ are inverses of each other.
\end{proof}

\begin{definition}
    Given a closed tree $T$. Let $\Omegacrsub(T,k)$ denote the full subcategory of $\Omegacrsub(T)$ consisting of objects $(I, X)$ such that  $|I| \leq k$, and  $X_e$ is a $k$-dendroidal tree for every $e \in I$.
\end{definition}
Given $k \geq 1$, the functor $q \colon \Omegacrsub(T) \to \cP(v(r))$ restricts to a functor 
\begin{equation}
    q_{k} \colon \Omegacrsub(T,k) \to \cP_{\leq k}(v(r)).
\end{equation}
By the same argument as \cref{lem:q-coCart}, we get:
\begin{lemma}\label{lem:qk-coCart}
   Given $k \geq 1$ and a closed tree $T$, the functor  $q_{k} \colon \Omegacrsub(T,k) \to \cP_{\leq k}(v(r))$ is a coCartesian fibration. 
\end{lemma}
By \cref{prop:Gamma-equivalence} we have:
\begin{corollary}
    Given $k \geq 1$ and  a closed tree $T$, the isomorphism $\Gamma(T)$ restricts to an isomorphism 
   \begin{equation}
    \Gamma(T,k) \colon \Omegacrsub(T,k) \simeq (\Omegacleqk)_\rsubT.
   \end{equation}
\end{corollary}

\subsection{The right adjoint $\Rk$}
In this subsection, we construct the fully faithful right adjoint $\Rk$.
First, we need two useful lemmas:
\begin{lemma}\label{lem:right-adjoint-limit-helper-1}
    Given a finite set $S$ and a space $X_s$ for each $s \in S$. Consider the functor $$\cF_X \colon (\cP(S))^{\op} \to \Spaces$$ defined by 
    taking a subset $I$ to $\prod_{s \in I} X_s$, and taking an inclusion $I \subset I'$ to the projection map 
    $$\prod_{s' \in I'} X_{s'} \to \prod_{s \in I} X_s$$ where we project away the $X_{s'}$ factors for $s' \notin I$. 
    Given $k \geq 1$, let $\cF_X|_{\leq k}$ be the composite 
    $$(\cP_{\leqk}(S))^{\op} \hookrightarrow \cP(S)^{\op}\xrightarrow{\cF_X} \Spaces.$$
    Then the
    canonical map 
    \begin{equation}\label{eq:RK-lemma-desired-eq}
        \prod_{s \in S}X_s = \cF_X(S) \to \lim(\cF_X|_{\leq k})
    \end{equation}
    is an equivalence.
\end{lemma}

\begin{proof}
    Note that $\cF_X$ is strongly Cartesian (see \cite[Definition 6.1.1.2]{HA}), that is,
    $\cF_X$ is the right Kan extension of $$\cF_X|_{\leq 1} \colon (\cP_{\leq 1}(S))^{\op} \to \Spaces$$ along the inclusion $(\cP_{\leq 1}(S))^{\op} \hookrightarrow \cP(S)^{\op}$. Thus for any $k \geq 1$,  $\cF_X$ is also the right Kan extension of 
    $\cF_X|_{\leq k} \colon (\cP_{\leqk}(S))^{\op} \to \Spaces$ along the inclusion $(\cP_{\leqk}(S))^{\op} \hookrightarrow \cP(S)^{\op}$. Finally, \eqref{eq:RK-lemma-desired-eq} is an equivalence by evaluating the right Kan extension at the object $S$.
    \end{proof}
    Given a functor of $\infty$-categories $p \colon \cC \to \cD$. For any $d \in \cD$ we denote by $\cC_d$ the fiber $p^{-1}(d)$.
    
    \begin{lemma}\label{lem:right-adjoint-limit-helper-2}
        Given a coCartesian fibration of $\infty$-categories $p \colon \cC \to \cD$ and a functor $\cF \colon \cC \to \Spaces$. For any $d \in \cD$, the canonical map 
        \begin{equation}
            \lim_{c \in \cC_d}\cF(c) \to \lim_{c \in \cC_{/d}}\cF(c)= \RKan\cF(d)
        \end{equation}
        is an equivalence.
    \end{lemma}
    \begin{proof}
       It suffices to show that for any $d \in D$, the inclusion functor $\cC_d \hookrightarrow \cC_{/d}$ is cofinal. This follows from the adjunction 
           \begin{equation}
               \begin{tikzcd}[column sep=1.5cm]
                   \cC_{/d}
                   \arrow[yshift=0.9ex]{r}{}
                   \arrow[hookleftarrow, yshift=-0.9ex]{r}[yshift=-0.15ex]{\bot}[swap]{}
                   &
                    \cC_d
                   \end{tikzcd}
           \end{equation}
           where the right adjoint is the canonical inclusion, and the left adjoint is given by the $p$-coCartesian factorization.
    \end{proof}
Throughout the rest of the subsection, we fix $1 \leq k \leq j \leq \infty$.
\begin{proposition}\label{prop:Rk-complete-Segal}
    The right Kan extension $\RKan_k \colon \cP(\Omegacleqk) \to \cP(\Omegacleq{j})$ preserves complete Segal presheaves.
\end{proposition}
\begin{proof}
    We first prove that $\RKan_k$ preserves Segal presheaves. Given a Segal presheaf $\cF \in \Seg(\Omegacleqk)$. By \cref{prop:closed-cutting-Segal}\eqref{item:closed-cutting-4}, it suffices to show to the following:
    for any closed $j$-dendroidal tree $T$ with root $r$,
    the canonical map
    \begin{equation}\label{eq:Rkan-desired-eq}
        \RKan_k\cF(T) \to  \RKan_k\cF(\CC{r}) \times_{\prod_{e \in v(r)} \RKan_k\cF(\overline{e})} \prod_{e \in v(r)}\RKan_k\cF(T_{\geq e})
    \end{equation}
    is an equivalence. 

    By \cref{lem:right-adjoint-1}, we see that the inclusion $(\Omegacleqk)_\rsubT \hookrightarrow (\Omegacleqk)_{/T}$ is cofinal. Thus, the induced map 
    \begin{equation}
        \RKan_k\cF(T) =  \lim_{X \in ((\Omegacleqk)_{/T})^\op} \cF(X) 
        \to \lim_{X \in ((\Omegacleqk)_\rsubT )^{\op}} \cF(X)
    \end{equation}
    is an equivalence. Once again, for brevity, we will omit the $\op$.

    Now we unpack \eqref{eq:Rkan-desired-eq} as a sequence of equivalences.
    \begin{equation}
        \begin{aligned}
            \RKan_k\cF(T) &\simeq \lim_{X \in (\Omegacleqk)_\rsubT} \cF(X) \\ 
            &\simeq \lim_{(I, X_I) \in \Omegacrsub(T,k)} \cF(\CC{I} \sqcup_{e \in I}X_e) \\
            &\simeq \lim_{I \in \cP_{\leqk}(v(r))}
            \left(\lim_{X_I \in \prod_{e\in I} (\Omegacleqk)_{^\rsub/T_{\geq e}}} \cF(\CC{I} \sqcup_{e \in I}X_e) \right)\\
            &\simeq 
            \lim_{I \in \cP_{\leqk}(v(r))} 
             \left(\lim_{X_I \in \prod_{e\in I} (\Omegacleqk)_{^\rsub/T_{\geq e}}}
            \cF(\CC{I}) \times_{\prod_{e \in I} \cF(\overline{e})}
            \prod_{e \in I}\cF(X_e)
            \right)\\
            &\simeq \lim_{I \in \cP_{\leqk}(v(r))} 
            \left(
            \cF(\CC{I}) \times_{\prod_{e \in I} \cF(\overline{e})}\,
            \prod_{e \in I}
            \lim_{X_e\in (\Omegacleqk)_{^\rsub/T_{\geq e}}} \cF(X_e) \right)\\
            &\simeq 
            \lim_{I \in \cP_{\leqk}(v(r))} 
            \left(
            \cF(\CC{I}) \times_{\prod_{e \in I} \cF(\overline{e})}\,
            \prod_{e \in I}
            \RKan_k\cF(T_{\geq e}) \right)\\
            &\simeq \left(\lim_{I \in \cP_{\leqk}(v(r))}
            \cF(\CC{I}) \right) \times_{\prod_{e \in I} \,\cF(\overline{e})}
            \left(\lim_{I \in \cP_{\leqk}(v(r))}
            \prod_{e \in I} 
            \RKan_k\cF(T_{\geq e}) \right)\\
            &\simeq \RKan_k\cF(\CC{r}) \times_{\prod_{e \in v(r)} \RKan_k\cF(\overline{e})} \prod_{e \in v(r)}\RKan_k\cF(T_{\geq e})
        \end{aligned}
    \end{equation}
    Here we used \cref{prop:Gamma-equivalence} for the second equivalence, 
    \cref{lem:qk-coCart} and \cref{lem:right-adjoint-limit-helper-2} for the third equivalence,  \cref{prop:closed-cutting-Segal}\eqref{item:closed-cutting-4} for the fourth equivalence.
    Finally, we used \cref{lem:right-adjoint-limit-helper-1} for the last equivalence.

    The proof of $\RKan_k$ preserving complete objects is the same as the left adjoint case.
\end{proof}

\begin{definition}\label{def:Rk-def}
    Let $\Rk$ denote the functor
    \begin{equation}
        \Opunleqk = \Segcpl(\Omegacleqk) \xhookrightarrow{\RKan_k} \Segcpl(\Omegacleq{j}) = \Opunleq{j}.
    \end{equation}
\end{definition}
Finally, we have our right adjoint statement:
\begin{theorem}\label{thm:right-adjoint-theorem}
    The adjunction \eqref{eq:RKan-k-adjunction}
    restricts to an adjunction 
    \begin{equation}\label{eq:Rk-adjunction}
    \begin{tikzcd}[column sep=1.5cm]
  (-)^k \colon \Opunleq{j}
\arrow[yshift=0.9ex]{r}{}
\arrow[hookleftarrow, yshift=-0.9ex]{r}[yshift=-0.15ex]{\bot}[swap]{}
&
\Opunleqk \colon \Rk 
\end{tikzcd}
    \end{equation}
    Furthermore, the right adjoint $\Rk$ is fully faithful. 
\end{theorem}

We have an explicit description of the multi-ary morphism spaces of $\cF$: 
\begin{corollary}\label{cor:right-adj-description-on-F}
    Let $\cF$ be a unital $k$-restricted $\infty$-operad. For $n \geq k$, the space  
    $\Rk \cF(\CC{n})$ is the limit 
    \begin{equation}
        \lim_{I \in (\cP_{\leq k}(v(r)))^\op} \cF(\CC{I}).
    \end{equation}
    Here $r$ is the root and $\CC{I}$ is the closed sub $|I|$-corolla of $\CC{n}$ corresponding to $I$. In particular, $\Rk\cF$ has the same colors as $\cF$.
    Given colors $X_1, \cdots, X_n, Y$  in $\cF$, we have 
    \begin{equation}
        \Mul_{\Rk \cF}(X_1, \cdots, X_n; Y) = 
        \lim_{(i_1, \cdots , i_l) \in (\cP_{\leq k}(v(r)))^\op} 
        \Mul_{\cF}(X_{i_1}, \cdots, X_{i_l}; Y).
    \end{equation}
\end{corollary}
Intuitively, $n$-ary morphisms of $\Lk \cF$ is the space of collections of $(\leq k)$-ary morphisms that are compatible under taking units.

By \cref{obs:k-restricted-op-eq}, we can detect when a unital $k$-restricted $\infty$-operad $\cF$ is in the image of $\Rk$:
\begin{corollary}\label{cor:right-adj-image-of-Rk}
    Let $\cF$ be a unital $j$-restricted $\infty$-operad. The following are equivalent:
    \begin{enumerate}
        \item $\cF$ is in the image of $\Rk$.
        \item For any $k < n \leq j$, the canonical map 
    \begin{equation}
    \cF(\CC{n}) \to  \lim_{I \in (\cP_{\leq k}(v(r)))^\op} \cF(\CC{I})
    \end{equation}
    is an equivalence.
        \item For any $k < n \leq j$ and colors $X_1, \cdots, X_n, Y$ in $\cF$, the canonical map 
        \begin{equation}
        \Mul_{\cF}(X_1, \cdots, X_n; Y) \to 
        \lim_{(i_1, \cdots , i_l) \in (\cP_{\leq k}(v(r)))^\op} 
        \Mul_{\cF}(X_{i_1}, \cdots, X_{i_l}; Y)
        \end{equation}
        is an equivalence.
    \end{enumerate}    
\end{corollary}

\begin{example}\label{ex:k-1-right-adjoint}
    Let $k = 1$. By \cref{ex:k-1-part-I} we see that $\Opunleq{1} \simeq \cat$. By \cref{cor:right-adj-description-on-F}, the fully faithful composite 
    \begin{equation}
        \cat \simeq \Opunleq{1} \xrightarrow{\mathrm{R}_1} \Op
    \end{equation}
    takes an $\infty$-category $\cC$ to the unital $\infty$-operad $\mathrm{R}_1 \cC$ with underlying $\infty$-category $\cC$ and 
    \begin{equation}
        \Mul_{\mathrm{R}_1 \cC}(X_1, \cdots, X_n; Y) = \prod_{i  = 1}^n \Hom_{\cC}(X_i, Y).
    \end{equation}
    This is the \emph{coCartesian} $\infty$-operad
    $\cC^{\sqcup}$ constructed in \cite[\S 2.4.3]{HA}. In particular, if the underlying $\infty$-category of $\cO$ is $\pt$, then $\mathrm{R}_1\cO \simeq \EE_\infty$.
\end{example}
Let us denote the localization functor $\Rk \circ (-)^k$ simply as $\Rk$.
Let $\cO$ be a unital $\infty$-operad, we have a co-filtration
\begin{equation}\label{eq:right-adjoint-co-fil}
    \cO \to  \cdots \to \Rk \cO \to \cdots \to \mathrm{R}_2 \cO \to \mathrm{R}_1 \cO
\end{equation}
The map $\cO \to \Rk \cO$ is an equivalence on $(\leq k)$-ary morphisms. This implies the following:
\begin{corollary}\label{cor:right-co-fil}
    Let $\cO$ be a unital $\infty$-operad, the canonical map 
    \begin{equation}
        \cO \to \lim_{k \in \mathbb{N}_{\geq 1}} \Rk \cO
    \end{equation}
    is an equivalence.
\end{corollary}

\begin{example}\label{ex:E1-right-adjoint}
   Fix $k \geq 1$. We would like to understand $\Rk\AA_\infty$. 
   Let us give an abstract description for the multi-ary morphism spaces of $\AA_{\infty}$: given a set $S$, let $\CC{S}$ be the closed $|S|$-corolla with the set of maximal edges being $S$. Then $\AA_{\infty}(\CC{S})$ is the set of linear ordering on $S$. Furthermore, given an inclusion $S \subset S'$, the induced map $\AA_{\infty}(\CC{S'}) \to \AA_{\infty}(\CC{S})$ is simply given by restricting the ordering to $S$.
   By \cref{cor:right-adj-description-on-F}, we see that the space $\Rk\AA_{\infty}(n) = \Rk\AA_{\infty}(\CC{n})$ is equivalent to the set of linear orderings $\sigma_I$, one for each $I \subset \{1, \cdots, n\}$ of size $\leq k$, that are compactible under restricting. The map $\AA_{\infty}(n) \to \Rk\AA_{\infty}(n)$ takes a total order on $\{1, \cdots, n\}$ to the compatible family of induced total order on subsets $I \subset \{1, \cdots, n\}$ with $|I|\leq k$.
   Since total orders are determined by its restriction on pairwise elements, we see that for $k > 2$, the map $\AA_{\infty}(n) \to \Rk\AA_{\infty}(n)$ is injective for any $n$.
   
    For $k = 1$, $\mathrm{R}_1\AA_{\infty}(n)$ is simply $\pt$, hence $\mathrm{R}_1\AA_{\infty} \simeq \EE_\infty$. Note this also follows from \cref{ex:k-1-right-adjoint}. Now $k = 2$, an element in $\mathrm{R}_2\AA_{\infty}(n)$ is corresponds to a choice of total orderings  on pairs $\{i, j\} \subset \{1, \cdots, n\}$. Therefore $|\mathrm{R}_2\AA_{\infty}(n)| = 2^{\frac{n(n-1)}{2}}$.
    For $k \geq 3$, we claim that the canonical map $\AA_{\infty} \to \mathrm{R}_k\AA_{\infty}$ is an equivalence. It suffices to show that it is surjective:
    Given a set $S$, the pairwise total ordering $\sigma_{x,y}$ defines an inequality $<_\sigma$ on $S$, and the transitivity property is guaranteed by the total ordering on the triplets. 
   \end{example}
    To summarize, we have the following:
    \begin{proposition}\label{prop:co-filtration-for-AA_infty}
       The co-filtration \eqref{eq:right-adjoint-co-fil} for $\AA_\infty$ is of the form 
       \begin{equation}
           \AA_\infty \xrightarrow{\simeq} \cdots \xrightarrow{\simeq} \mathrm{R}_3\AA_\infty \to \mathrm{R}_2 \AA_\infty \to \mathrm{R}_1 \AA_\infty \simeq \EE_\infty.
       \end{equation}
    \end{proposition}

\bibliographystyle{amsalpha}
\bibliography{bib.bib}

\providecommand{\bysame}{\leavevmode\hbox to3em{\hrulefill}\thinspace}
\providecommand{\MR}{\relax\ifhmode\unskip\space\fi MR }
\providecommand{\MRhref}[2]{%
  \href{http://www.ams.org/mathscinet-getitem?mr=#1}{#2}
}
\providecommand{\href}[2]{#2}
\begin{thebibliography}{CHH18}

\bibitem[Bar18]{Barwick_2018}
Clark Barwick, \emph{From operator categories to higher operads}, Geometry \& Topology \textbf{22} (2018), no.~4, 1893--1959.

\bibitem[Bar22]{barkan2023}
Shaul Barkan, \emph{Arity approximation of $\infty$-operads}, arXiv preprint arXiv:2207.07200 (2022).

\bibitem[CH20]{MR4038556}
Hongyi Chu and Rune Haugseng, \emph{Enriched {$\infty$}-operads}, Adv. Math. \textbf{361} (2020), 106913, 85. \MR{4038556}

\bibitem[CHH18]{two-models}
Hongyi Chu, Rune Haugseng, and Gijs Heuts, \emph{Two models for the homotopy theory of infinity-operads}, Journal of Topology \textbf{11} (2018), no.~4, 857–873.

\bibitem[CM13]{MR3100887}
Denis-Charles Cisinski and Ieke Moerdijk, \emph{Dendroidal {S}egal spaces and {$\infty$}-operads}, J. Topol. \textbf{6} (2013), no.~3, 675--704. \MR{3100887}

\bibitem[G{\"o}p18]{goppl}
Florian G{\"o}ppl, \emph{A spectral sequence for spaces of maps between operads}, arXiv preprint arXiv:1810.05589 (2018).

\bibitem[Heu21]{MR4320769}
Gijs Heuts, \emph{Goodwillie approximations to higher categories}, Mem. Amer. Math. Soc. \textbf{272} (2021), no.~1333, ix+108. \MR{4320769}

\bibitem[JT07]{MR2342834}
Andr\'e{} Joyal and Myles Tierney, \emph{Quasi-categories vs {S}egal spaces}, Categories in algebra, geometry and mathematical physics, Contemp. Math., vol. 431, Amer. Math. Soc., Providence, RI, 2007, pp.~277--326. \MR{2342834}

\bibitem[KK24]{KK}
Alexander Kupers and Manuel Krannich, \emph{The {$\infty$}-operadic foundations for embedding calculus}, arXiv preprint arXiv:2409.10991 (2024).

\bibitem[Lur09]{HTT}
Jacob Lurie, \emph{Higher topos theory}, Annals of Mathematics Studies, vol. 170, Princeton University Press, Princeton, NJ, 2009. \MR{2522659}

\bibitem[Lur17]{HA}
\bysame, \emph{Higher algebra}, 2017, Available at the author's website (version dated 9/18/2017).

\bibitem[Moe21]{MR4222648}
Ieke Moerdijk, \emph{Closed dendroidal sets and unital operads}, Theory Appl. Categ. \textbf{36} (2021), Paper No. 5, 118--170. \MR{4222648}

\bibitem[Rez01]{rezk2001model}
Charles Rezk, \emph{A model for the homotopy theory of homotopy theory}, Transactions of the American Mathematical Society \textbf{353} (2001), no.~3, 973--1007.

\bibitem[Sta63]{stasheff1963homotopy}
James~Dillon Stasheff, \emph{Homotopy associativity of h-spaces. ii}, Transactions of the American Mathematical Society \textbf{108} (1963), no.~2, 293--312.

\end{thebibliography}

\end{document}